\documentclass[12pt,oneside]{article}
\usepackage[inner=32mm,outer=31mm,tmargin=30mm,bmargin=40mm]{geometry}
\usepackage[utf8]{inputenc}
\usepackage[T1]{fontenc}
\usepackage{amssymb}
\usepackage{hyperref}
\usepackage{geometry}
\usepackage{amsmath}
\usepackage{amsthm}
\usepackage{xfrac}
\usepackage{graphicx}
\usepackage{color}
\usepackage[normalem]{ulem}
\usepackage[english]{babel}
\usepackage[onehalfspacing]{setspace}
\usepackage{verbatim}

\newcommand{\IR}{\ensuremath{\mathbb{R}}}
\newcommand{\IN}{\ensuremath{\mathbb{N}}}

\newcommand{\C}{\mathcal{C}}

\renewcommand{\d}{{\rm d}}

\DeclareMathOperator{\cost}{cost}

\newcommand{\dist}{{\rm dist}}

\newcommand{\vol}{{\rm vol}}

\newcommand{\set}[1]{\left\{#1\right\}}
\newcommand{\abs}[1]{\left|#1\right|}
\newcommand{\brackets}[1]{\left(#1\right)} 

\newcommand{\norm}[1]{\left\Vert#1\right\Vert}

\newtheorem{thm}{Theorem}
\newtheorem{cor}{Corollary}
\theoremstyle{plain}
\newtheorem{lemma}{Lemma}

\theoremstyle{definition}

\newtheorem{rem}{Remark}
\newtheorem*{ack}{Acknowledgements}

\title{Uniform recovery of high-dimensional $C^r$-functions}

\author{
David Krieg
\\ 
Mathematisches Institut, Universit\"at Jena\\ 
Ernst-Abbe-Platz 2, 07743 Jena, Germany,  \\ 
david.krieg@uni-jena.de}
\date{\today}

\begin{document}

\maketitle

\begin{abstract} \noindent
 We consider functions on the $d$-dimensional
 unit cube whose partial derivatives up to order $r$
 are bounded by one.
 It is known that the minimal number of function values
 that is needed to approximate the integral of such functions 
 up to the error $\varepsilon$ 
 is of order
 $(d/ \varepsilon)^{d/r}$.
 Among other things,
 we show that the minimal number of function values that is needed to
 approximate such functions in the uniform norm
 is of order $(d^{r/2} /\varepsilon)^{d/r}$
 whenever $r$ is even.
\end{abstract}

\section{Introduction and results}

We study the problem of the uniform recovery of
functions by deterministic algorithms
that use a finite number of function values.
We are interested in the class
\begin{equation}
\label{standard classes}
 \C^r_d= \set{f \in \C^r\brackets{[0,1]^d} \mid 
 \Vert D^\beta f \Vert_\infty \leq 1 
 \text{\ for all } \beta \in \IN_0^d \text{ with  } \abs{\beta}\leq r}
\end{equation}
of real-valued functions on the $d$-dimensional unit cube 
whose partial derivatives up to order $r\in\IN$
are continuous and bounded by one.
It is well known that the integration of functions from $\C^r_d$
suffers from the curse of dimensionality.
In fact, the minimal number $n^{\rm int}(\varepsilon,\C^r_d)$
of function values that is needed to guarantee an integration error
$\varepsilon \in (0,1/2)$ for any function from $\C^r_d$
grows super-exponentially with the dimension.
It is proven in \cite{hnuw} that there are positive constants $c_r$
and $C_r$ such that
\begin{equation*}
 \brackets{c_r\, d^{1/r} \varepsilon^{-1/r}}^d
 \leq n^{\rm int}(\varepsilon,\C^r_d)
 \leq \brackets{C_r\, d^{1/r} \varepsilon^{-1/r}}^d
 \end{equation*}
for all $\varepsilon \in (0,1/2)$ and $d\in\IN$.
Roughly speaking $n^{\rm int}(\varepsilon,\C^r_d)$
is of order $(d/ \varepsilon)^{d/r}$.
See Section~\ref{setting section} for a precise definition 
of the $n^{\rm int}(\varepsilon,\C^r_d)$
and further notation.

Since an $\varepsilon$-approximation of the function
immediately yields an $\varepsilon$-approximation of its 
integral, the uniform recovery
of functions from $\C^r_d$ can only be harder.
But how hard is the uniform recovery problem?
Is it significantly harder than the integration problem?
These questions were recently posed in \cite[Section~6]{abc}.

If $r=1$, the answer is known.
In this case, the minimal number 
$n^{\rm app}(\varepsilon,\C^r_d)$
of function values that is needed to guarantee an approximation error
$\varepsilon>0$ for any function from $\C^r_d$
in the uniform norm behaves similarly to $n^{\rm int}(\varepsilon,\C^r_d)$.
There are positive constants $c$
and $C$ such that
\begin{equation*}
 \brackets{c\, d\, \varepsilon^{-1}}^d
 \leq n^{\rm app}(\varepsilon,\C^1_d) \leq
 \brackets{C\, d\, \varepsilon^{-1}}^d
\end{equation*}
for all $\varepsilon \in (0,1/2)$ and $d\in\IN$.
This result is basically contained in \cite{suk}.
Nonetheless, we will present its proof.
If $r\geq 2$ is even, we obtain the following result.

\begin{thm}
\label{main theorem}
Let $r\in\IN$ be even. 
Then there are positive constants $c_r$, $C_r$ and $\varepsilon_r$
such that
\begin{equation*}
 \brackets{c_r \sqrt{d}\, \varepsilon^{-1/r}}^d
 \leq n^{\rm app}(\varepsilon,\C^r_d) \leq
 \brackets{C_r \sqrt{d}\, \varepsilon^{-1/r}}^d
\end{equation*}
for all $d\in\IN$ and $\varepsilon\in (0,\varepsilon_r)$.
The upper bound holds for all $\varepsilon>0$.
\end{thm}

Roughly speaking $n^{\rm app}(\varepsilon,\C^r_d)$
is of order $(d^{r/2} /\varepsilon)^{d/r}$.
If the error tolerance $\varepsilon$ is fixed,
the complexity grows like $d^{d/2}$.
This is in contrast to the case $r=1$,
where we have a growth of order $d^d$.
If $r\geq 3$ is odd, 
we only have a partial result.


\begin{thm}
\label{main theorem odd}
Let $r\geq 3$ be odd. 
Then there are positive constants $c_r$, $C_r$ and $\varepsilon_r$
such that
\begin{equation*}
 \brackets{c_r \sqrt{d}\, \varepsilon^{-1/r}}^d
 \leq n^{\rm app}(\varepsilon,\C^r_d) \leq
 \brackets{C_r\, d^{\frac{r+1}{2r}} \varepsilon^{-1/r}}^d
\end{equation*}
for all $d\in\IN$ and $\varepsilon\in (0,\varepsilon_r)$.
The upper bound holds for all $\varepsilon>0$.
\end{thm}

We point to the fact that 
$n^{\rm app}(\varepsilon,\C^r_d) \leq n^{\rm app}(\varepsilon,\C^{r-1}_d)$
since the upper bound resulting 
from Theorem~\ref{main theorem}
may improve on the upper bound 
of Theorem~\ref{main theorem odd}
for $d\succ \varepsilon^{-2/(r-1)}$
if $r\geq 3$ is odd.
In~this case, we do not know the exact behavior of
$n^{\rm app}(\varepsilon,\C^r_d)$ as a function
of both $d$ and $\varepsilon$.
If regarded as a function of $\varepsilon$,
the complexity is of order $\varepsilon^{-d/r}$.
If regarded as a function of $d$, 
it is of order $d^{d/2}$.

Altogether, our results justify the following comparison.

\begin{cor}
 The uniform recovery problem on the class $\C^r_d$ is
 significantly harder than the integration problem 
 if and only if $r\geq 3$.
\end{cor}

Except for the case $r=1$, the lower bounds of the
previous theorems even hold for the smaller class
\begin{equation}
\label{tilde classes}
 \widetilde\C^r_d= \set{f \in \C^r\brackets{[0,1]^d} \mid 
 \Vert \partial_{\theta_1}\cdots\partial_{\theta_\ell} f \Vert_\infty \leq 1 
 \text{\ for all } \ell \leq r \text{ and  } \theta_i\in S_{d-1}}
\end{equation}
of functions whose directional derivatives up to order $r\in\IN$
are bounded by one.
For this class, we obtain sharp bounds on the $\varepsilon$-complexity
of the uniform recovery problem for any $r\in\IN$.
The minimal number $n^{\rm app}(\varepsilon,\widetilde\C^r_d)$ of
function values that is needed to guarantee
an approximation error $\varepsilon$ for every function
from $\widetilde\C^r_d$ in the uniform norm satisfies the following.

\begin{thm}
\label{side theorem}
Let $r\in\IN$. 
There are positive constants $c_r$, $C_r$ and $\varepsilon_r$
such that
\begin{equation*}
 \brackets{c_r \sqrt{d}\, \varepsilon^{-1/r}}^d
 \leq n^{\rm app}(\varepsilon,\widetilde\C^r_d) \leq
 \brackets{C_r \sqrt{d}\, \varepsilon^{-1/r}}^d
\end{equation*}
for all $d\in\IN$ and $\varepsilon\in (0,\varepsilon_r)$.
The upper bound holds for all $\varepsilon>0$.
\end{thm}

Before we turn to the proofs,
we shortly discuss some related problems.

\begin{rem}[Global optimization]
 We obtain analogous estimates
 for the complexity of global optimization on $F=\C^r_d$
 or $F=\widetilde\C^r_d$.
 This is because
 the minimal number $n^{\rm opt}(\varepsilon,F)$ 
 of function values that is needed to
 guarantee an $\varepsilon$-approximation of the maximum
 of a function from $F$ satisfies \cite{n_opt,w}
 \begin{equation*}
  n^{\rm app}(2\varepsilon,F)
  \leq
  n^{\rm opt}(\varepsilon,F)
  \leq
  n^{\rm app}(\varepsilon,F).
  \end{equation*}
\end{rem}

\begin{rem}[Infinite smoothness]
 It is proven in \cite{nw_inf}
 that even the uniform recovery of functions from
 \begin{equation*}
  \C^\infty_d = 
  \set{f \in \C^\infty\brackets{[0,1]^d} \mid 
 \Vert D^\beta f \Vert_\infty \leq 1 
 \text{\ for all } \beta \in \IN_0^d}
 \end{equation*}
 suffers from the curse of dimensionality.
 For $r\in\IN$, we have seen that the complexity $n^{\rm app}(\varepsilon,\C^r_d)$
 in fact depends super-exponentially on the dimension.
 It would be interesting to verify whether
 this is also true for $r=\infty$.
 We remark that the uniform recovery problem does not suffer from the curse
 if the target function lies within the modified class
 \begin{equation*}
  \overline{\C^\infty_d} = 
  \set{f \in \C^\infty\brackets{[0,1]^d} \mid 
 \sum_{\abs{\beta}=k} \frac{\Vert D^\beta f \Vert_\infty}{\beta !} \leq 1 
 \text{\ for all } k \in \IN_0}
 \end{equation*}
 of smooth functions. This is proven in \cite{vib}.
\end{rem}

\begin{rem}[Algorithms]
 This paper is not concerned with explicit algorithms.
 Nonetheless, our proof shows that there are optimal algorithms in the sense
 of Theorem~\ref{main theorem}, \ref{main theorem odd} and 
 \ref{side theorem} whose information is given by function values
 at a regular grid and small clouds around the grid points.
 This information can be used for a subcubewise Taylor approximation
 of the target function around the grid points,
 where the partial derivatives of order less than $r$ 
 are replaced by divided differences.
 The resulting algorithm is indeed optimal for the class $\widetilde \C^r_d$.
 However, the author does not know whether it is also optimal for $\C^r_d$.
\end{rem}

\begin{rem}[Other domains]
 Our lower bounds are still valid, 
 if the domains $[0,1]^d$ are replaced
 by any other sequence of domains $D_d\subset\IR^d$ 
 that satisfies $\vol_d(D_d)\geq a^d$ for some $a>0$
 and all $d\in\IN$.
 The upper bounds, however, heavily exploit the geometry of the unit cube.
 We remark that the curse of dimensionality for 
 the integration problem on general domains
 is studied in the recent paper~\cite{hpu}.
\end{rem}

\begin{rem}[Integration for $\widetilde\C^r_d$]
 Note that the right behavior 
 of the complexity $n^{\rm int}(\varepsilon,\widetilde\C^r_d)$ 
 of the integration problem for $\widetilde\C^r_d$
 as a function of $d$ and $\varepsilon$ is still open.
\end{rem}

 \section{The setting}
\label{setting section}

Let $d\in\IN$ and let $F$ 
be a class of continuous real-valued functions on
$[0,1]^d$.
We study the problem of uniform approximation
on $F$ via function values in the worst case setting.
An algorithm for numerical approximation
is a mapping $A=\varphi\circ N$
built from an information map $N:F\to \IR^n$
for some $n\in\IN$ and an arbitrary map 
$\varphi:\IR^n\to L_\infty\brackets{[0,1]^d}$.
The information map is of the form
\begin{equation*}
 N(f)=\brackets{f(x_1),\hdots,f(x_n)},
\end{equation*}
where the points $x_i\in[0,1]^d$ may be chosen
based on the already computed function values
$f(x_1),\hdots,f(x_{i-1})$ for $i=1,\hdots,n$.
The cost of the algorithm is the number $n$
of computed function values and denoted by $\cost(A)$.
Its worst case error is the quantity
\begin{equation*}
 e^{\rm app}(A,F) = \sup\limits_{f\in F} \norm{f-A(f)}_\infty.
\end{equation*}
See Novak and Woźniakowski~\cite[Chapter~4]{trac1} for a 
detailed discussion of algorithms
and their errors and cost in various settings.

The $n$th minimal worst case error
is the smallest worst case error 
of algorithms using at most $n$ function values,
that is
\begin{equation*}
 e^{\rm app}(n,F)=\inf\set{e^{\rm app}(A,F) \mid 
 \cost(A)\leq n}.
\end{equation*}
Finally, we formally define 
the minimal number 
of function values needed to approximate
an unknown function from $F$ 
up to the error $\varepsilon>0$ in the
uniform norm as
\begin{equation*}
 n^{\rm app}(\varepsilon,F)
 = \min\set{n\in\IN_0 \mid e^{\rm app}(n,F) \leq \varepsilon }.
\end{equation*}

Our results are concerned with the classes $F=\C^r_d$
and $F=\widetilde\C^r_d$ for $r\in\IN_0$ and $d\in\IN$
as defined in \eqref{standard classes}
and \eqref{tilde classes}.
Here, $D^{\beta}$ denotes the partial derivative
of order $\beta\in\IN_0^d$ and $\abs{\beta}=\sum_{i=1}^d \beta_i$.
Moreover, $\partial_\theta$ denotes the directional derivative
in the direction $\theta\in S_{d-1}$,
where $S_{d-1}$ is the euclidean unit sphere in $\IR^d$.
These classes are convex and symmetric.
Our proofs are based on the following fact.

\begin{lemma}[see~\cite{cw}]
\label{Bakhvalov}
Let $d\in\IN$ and $F\subset \C\brackets{[0,1]^d}$
be convex and symmetric. Then
 \begin{equation*}
  e^{\rm app}\brackets{n,F}
  = \inf\limits_{\substack{P\subset [0,1]^d\\\abs{P}\leq n}}\sup\limits_{\substack{f\in F\\f\mid_P=0}} \norm{f}_\infty
 .\end{equation*}
\end{lemma}

Under the assumptions of Lemma~\ref{Bakhvalov},
we even know that linear algorithms are optimal.
That is, for each $n\in\IN$, there are functions 
$g_1,\hdots,g_n \in L_\infty\brackets{[0,1]^d}$
and points $x_1,\hdots,x_n \in [0,1]^d$
such that the algorithm
\begin{equation*}
 A_n^*:F\to L_\infty\brackets{[0,1]^d},
 \qquad
 A_n^*(f)= \sum_{i=1}^n f(x_i) g_i
\end{equation*}
satisfies
\begin{equation*}
 e^{\rm app}(A_n^*,F)= e^{\rm app}(n,F).
\end{equation*}
These results go back to Bakhvalov~\cite{b}
and Smolyak~\cite{s}.
We refer to Creutzig and Wojtaszczyk~\cite{cw} for a proof.

We also talk about numerical integration on $F$
in the worst case setting.
Similarly to numerical approximation,
an algorithm for numerical integration
is a functional $A=\varphi\circ N$
built from an information map $N:F\to \IR^n$
like above and an arbitrary map 
$\varphi:\IR^n\to \IR$.
Its cost is $n$ and
its worst case error is
\begin{equation*}
 e^{\rm int}(A,F) = \sup\limits_{f\in F} 
 \abs{A(f) - \int_{[0,1]^d} f(x)~\d x}.
\end{equation*}

The $n$th minimal worst case error
is the smallest worst case error 
of algorithms using at most $n$ function values,
that is
$$
 e^{\rm int}(n,F)=\inf\set{e^{\rm int}(A,F) \mid 
 \cost(A)\leq n}.
$$
The minimal number 
of function values that is needed to guarantee
an $\varepsilon$-approximation 
of the integral of a function from $F$ 
is formally defined as
\begin{equation*}
 n^{\rm int}(\varepsilon,F)
 = \min\set{n\in\IN_0 \mid e^{\rm int}(n,F) \leq \varepsilon }.
\end{equation*}

\section{Upper bounds}

To estimate $e^{\rm app}\brackets{n,\C^r_d}$ from above, 
Lemma~\ref{Bakhvalov} says that
we can choose any point set $P$
with cardinality at most $n$ and 
give an upper bound on the 
maximal value of a function $f\in\C^r_d$
that vanishes on $P$.
In fact, we can choose any 
point set $Q$ with cardinality at most $n/(d+1)^{r-1}$
and assume that not only $f$ but all its derivatives
of order less than $r$ are arbitrarily 
small on $Q$.
More precisely, for any $\delta>0$, any $r\in\IN_0$, $d\in\IN$ and $Q\subset [0,1]^d$,
we define the subclasses
\begin{equation*}
 \C^r_d(Q,\delta) =
 \set{f\in\C^r_d \mid \abs{D^\alpha f(x)}\leq \delta^{2^{r-\abs{\alpha}-1}}
 \text{ for all } x\in Q \text{ and } \abs{\alpha} < r}
\end{equation*}
and the auxiliary quantities
\begin{equation*}
 E\brackets{Q,\C^r_d,\delta}=
 \sup\limits_{f\in \C^r_d(Q,\delta)} \norm{f}_\infty
 \qquad\text{and}\qquad
 E\brackets{Q,\C^r_d}= 
 \lim\limits_{\delta\downarrow 0} E\brackets{Q,\C^r_d,\delta}
\end{equation*}
and obtain the following.

\begin{lemma}
\label{small derivatives}
Let $d\in \IN$, $r\in\IN$ and $n\in \IN_0$.
If the cardinality of $Q\subset [0,1]^d$
is at most $n/(d+1)^{r-1}$, then
 \begin{equation*}
  e^{\rm app}\brackets{n,\C^r_d} \leq E\brackets{Q,\C^r_d}.
 \end{equation*}
\end{lemma}


\begin{proof}
 Let $\delta \in (0,1)$.
 We will construct a point set $P\subset [0,1]^d$
 with cardinality at most $n$ such that
 any $f\in\C^r_d$ with $f\vert_P=0$ is contained
 in $\C^r_d(Q,\delta)$.
 Then Lemma~\ref{Bakhvalov} yields
 \begin{equation*}
  e^{\rm app}\brackets{n,\C^r_d} \leq
  \sup\limits_{f\in \C^r_d:\, f\vert_P=0} \norm{f}_\infty
  \leq \sup\limits_{f\in \C^r_d(Q,\delta)} \norm{f}_\infty
  = E\brackets{Q,\C^r_d,\delta}.
 \end{equation*}
 Letting $\delta$ tend to zero yields the statement.
 
 If $r=1$, we can choose $P=Q$.
 Let us start with the case $r=2$.
 Given a set $M\subset [0,1]^d$ and $h\in (0,1/2]$, we define
 \begin{equation*}
  M[h]= M 
  \cup\bigcup_{\substack{(x,j)\in M\times\set{1\dots d}\\x+he_j\in [0,1]^d}} \set{x+he_j}\
  \cup\bigcup_{\substack{(x,j)\in M\times\set{1\dots d}\\x+he_j\not\in [0,1]^d}} \set{x-he_j}
 .\end{equation*}
 Obviously, the cardinality of $M[h]$ is at most $(d+1)\abs{M}$.
 Furthermore, we have 
 \begin{equation}
 \label{mean value}
   f\in\C^2_d\text{ with }\abs{f}\leq h^2\text{ on }M[h]\quad
   \Rightarrow\quad
    \abs{\frac{\partial f}{\partial x_j}} \leq 3h\text{ on }M\text{ for }j=1\dots d
 .\end{equation}
 This is a simple consequence of the mean value theorem:
 For any $j\in\set{1\dots d}$ and $x\in M$ with $x+he_j\in[0,1]^d$ there is some $\eta\in (0,h)$ with
 \begin{equation*}
  \abs{\frac{\partial f}{\partial x_j}\brackets{x+\eta e_j}}
  = \abs{\frac{f\brackets{x+h e_j}-f(x)}{h}}
  \leq 2h.
 \end{equation*}
 The same estimate holds for some $\eta\in(-h,0)$, if $x+he_j\not\in[0,1]^d$.
 The fundamental theorem of calculus yields
 \begin{equation*}
  \abs{\frac{\partial f}{\partial x_j}\brackets{x}}
  \leq \abs{\frac{\partial f}{\partial x_j}\brackets{x+\eta e_j}}
  + \abs{\eta}\cdot \max\limits_{\abs{t}\leq\eta} \abs{\frac{\partial^2 f}{\partial x_j^2}\brackets{x+te_j}}
  \leq 3h
 .\end{equation*}
 This means that we can choose $P=Q\left[\delta/3\right]$.
 
 For $r>2$ we repeat this procedure $r-1$ times.
 We use the notation
 \begin{equation*}
  M\left[h_1,\dots,h_{i}\right]=M\left[h_1,\dots,h_{i-1}\right]\left[h_i\right]
 \end{equation*}
 for $i>1$. We choose the point set
 $$
 P=Q\left[h_1,\dots,h_{r-1}\right], \quad \text{where} \quad
 h_i=3(\delta/9)^{2^{i-1}}
 $$
 for $i=1\hdots r-1$.
 Note that $3h_i=h_{i-1}^2$ for each $i\geq 2$.
 Clearly, the cardinality of $P$ is at most $(d+1)^{r-1}\abs{Q}$
 and hence bounded by $n$.
 Let $f\in\C^r_d$ vanish on $P$ and let 
 $\frac{\partial^\ell f}{\partial x_{j_1}\dots\partial x_{j_\ell}}$
 be any derivative of order $\ell <r$.
 Fact (\ref{mean value}) yields:\newline
 \begin{align*}
   &f\in C^r_d
   &&\text{ with }\abs{f}=0\leq h_{r-1}^2
   &&\text{ on }Q\left[h_1,\dots,h_{r-1}\right]\\
   \Rightarrow\
   &\frac{\partial f}{\partial x_{j_1}}\in C^{r-1}_d
   &&\text{ with }\abs{\frac{\partial f}{\partial x_{j_1}}}\leq 3h_{r-1} = h_{r-2}^2
   &&\text{ on }Q\left[h_1,\dots,h_{r-2}\right]\\
   \Rightarrow\
   &\frac{\partial^2 f}{\partial x_{j_1}\partial x_{j_2}}\in C^{r-2}_d
   &&\text{ with }\abs{\frac{\partial^2 f}{\partial x_{j_1}\partial x_{j_2}}}\leq 3h_{r-2}
   &&\text{ on }Q\left[h_1,\dots,h_{r-3}\right]\\
   \Rightarrow\ & &&\hdots &&\\
   \Rightarrow\
   &\frac{\partial^\ell f}{\partial x_{j_1}\dots\partial x_{j_\ell}}\in C^{r-\ell}_d
   &&\text{ with }\abs{\frac{\partial^\ell f}{\partial x_{j_1}\dots\partial x_{j_\ell}}}\leq 3h_{r-\ell}
   &&\text{ on }Q\left[h_1,\dots,h_{r-\ell-1}\right]
 .\end{align*}
 Since $Q$ is contained in $Q\left[h_1,\dots,h_{r-\ell-1}\right]$ and 
 $3h_{r-\ell}\leq \delta^{2^{r-\ell-1}}$, the lemma is proven.
\end{proof}

We can prove the desired upper bounds on $e\brackets{n,\C^r_d}$
by choosing $Q$ as a regular grid. 
We set
\begin{equation*}
 Q_m^d=\set{0,1/m,2/m,\hdots,1}^d
\end{equation*}
for $m\in\IN$. The following recursive formula is crucial.

\begin{lemma}
 \label{two step induction}
 Let $m\in \IN$, $d\geq 2$ and $r\geq 2$. Then
 \begin{equation*}
  E\brackets{Q_m^d,\C^r_d}
  \leq E\brackets{Q_m^{d-1},\C^r_{d-1}} + 
  \frac{1}{8m^2}\, E\brackets{Q_m^d,\C^{r-2}_d}.
 \end{equation*}
\end{lemma}

\begin{proof}
 We will prove for any $\delta>0$ that
 \begin{equation}
 \label{recursive formula}
  E\brackets{Q_m^d,\C^r_d,\delta}
  \leq E\brackets{Q_m^{d-1},\C^r_{d-1},\delta} + 
  \frac{1}{8m^2}\, E\brackets{Q_m^d,\C^{r-2}_d,\delta}.
 \end{equation}
 Letting $\delta$ tend to zero yields the statement.
  
 Let $f\in \C^r_d\brackets{Q_m^d,\delta}$.
 We need to show that $\norm{f}_\infty$ is bounded
 by the right hand side of \eqref{recursive formula}.
 Since $f$ is continuous, there is some $z\in[0,1]^d$
 such that $\abs{f(z)}=\norm{f}_\infty$.
 We distinguish two cases.
 
 If $z_d\in \set{0,1}$,
 the restriction $f\vert_{H}$ of $f$ to the hyperplane
 \begin{equation*}
  H=\set{x\in[0,1]^d \mid x_d=z_d}
 \end{equation*}
 is contained in $\C^r_{d-1}\brackets{Q_m^{d-1},\delta}$.
 This implies that
 \begin{equation*}
  \abs{f(z)}=\norm{f\vert_H}_\infty \leq E\brackets{Q_m^{d-1},\C^r_{d-1},\delta}
 \end{equation*}
 and the statement is proven.
 
 Let us now assume that $z_d\in (0,1)$.
 Then we have $\frac{\partial f}{\partial x_d}(z)=0$.
 We choose $y\in [0,1]^d$ such that $y_j=z_j$ for $j<d$
 and $y_d\in Q_m$ with $\abs{y_d-z_d}\leq 1/(2m)$.
 The restriction $f\vert_{H_0}$ of $f$ to the hyperplane
 \begin{equation*}
  H_0=\set{x\in[0,1]^d \mid x_d=y_d}
 \end{equation*}
 is contained in $\C^r_{d-1}\brackets{Q_m^{d-1},\delta}$.
 This implies that
 \begin{equation*}
  \abs{f(y)}=\norm{f\vert_{H_0}}_\infty \leq E\brackets{Q_m^{d-1},\C^r_{d-1},\delta}.
 \end{equation*}
 Moreover, the second derivative $\frac{\partial^2 f}{\partial x_d^2}$
 is contained in $\C^{r-2}_d\brackets{Q_m^d,\delta}$ and hence
 \begin{equation*}
  \norm{\frac{\partial^2 f}{\partial x_d^2}}_\infty \leq 
  E\brackets{Q_m^d,\C^{r-2}_d,\delta}.
 \end{equation*}
 By Taylor's theorem, there is some $\xi$ on the line segment between $y$ and $z$
 such that 
 \begin{equation*}
  f(y)= f(z) + \frac{1}{2} \frac{\partial^2 f}{\partial x_d^2}(\xi)
  \cdot (y_d-z_d)^2.
 \end{equation*}
 We obtain
 \begin{equation*}
 \begin{split}
  \abs{f(z)}\, &\leq\, \abs{f(y)} \,+\,
  \frac{(y_d-z_d)^2}{2}\cdot \norm{\frac{\partial^2 f}{\partial x_d^2}}_\infty \\
  &\leq\, E\brackets{Q_m^{d-1},\C^r_{d-1},\delta} \,+\, 
  \frac{1}{8m^2}\cdot E\brackets{Q_m^d,\C^{r-2}_d,\delta},
 \end{split}
 \end{equation*}
 as it was to be proven.
\end{proof}

By a double induction on $r$ and $d$ we obtain the following
result for even $r$.

\begin{lemma}
 \label{even r lemma}
 Let $d\in\IN$, $m\in\IN$ and $r\in\IN_0$ be even. Then
 \begin{equation*}
  E\brackets{Q_m^d,\C^r_d} \leq
  \frac{e d^{r/2}}{(2m)^r}.
 \end{equation*}
\end{lemma}

\begin{proof}
 We give a proof by induction on $d$.
 Let $\delta>0$ and $f\in \C^r_1\brackets{Q_m,\delta}$
 for some even number $r$.
 Since $f$ is continuous, there is some $z\in[0,1]$
 such that $\abs{f(z)}=\norm{f}_\infty$.
 Let $y\in Q_m$ with $\abs{y-z}<1/(2m)$.
 By Taylor's theorem, there is some $\xi$
 between $y$ and $z$ such that
 $$
 f(z)=\sum_{k=0}^{r-1}\frac{f^{(k)}(y)}{k!} (z-y)^k 
 + \frac{f^{(r)}(\xi)}{r!} (z-y)^r .
 $$
 Using that $\vert f^{(k)}(y)\vert\leq \delta^{2^{r-k-1}} \leq \delta^{r-k}$,
 we obtain for $\delta \leq 1/(2m)$ that
 $$
 \norm{f}_\infty \leq \sum_{k=0}^r\frac{\delta^{r-k}}{k!} \brackets{\frac{1}{2m}}^k 
 \leq \brackets{\frac{1}{2m}}^r \sum_{k=0}^r\frac{1}{k!}
 \leq \frac{e}{(2m)^r}.
 $$
 Since this is true for any such $f$ and any $\delta\leq 1/(2m)$,
 this proves the case $d=1$.
 
 Let now $d\geq 2$. We assume that the statement holds for every dimension
 smaller than $d$.
 To show that it also holds in dimension $d$,
 we use induction on $r$.
 For $r=0$ the statement is trivial since $E(Q_m^d,\C^0_d)=1$.
 Let $r\geq 2$ be even and assume that the statement
 holds in dimension $d$ for any even smoothness smaller than $r$.
 Lemma~\ref{two step induction} yields
 \begin{align*}
  E\brackets{Q_m^d,\C^r_d}
  \leq \frac{e (d-1)^{r/2}}{(2m)^r} + 
  \frac{1}{8m^2}\, \frac{e d^{r/2-1}}{(2m)^{r-2}}\\
  = \frac{e d^{r/2}}{(2m)^r}
  \brackets{\brackets{1-\frac{1}{d}}^{r/2}+\frac{1}{2d}}
  \leq \frac{e d^{r/2}}{(2m)^r} ,
 \end{align*}
 which completes the inner and therefore the outer induction.
\end{proof}

This immediately yields the upper bound of 
Theorem~\ref{main theorem}.

\begin{proof}[Proof of Theorem~\ref{main theorem} (Upper Bound)]
 Let $d\in\IN$, $r\in\IN$ be even and $\varepsilon >0$. We set
 $$
 n=(d+1)^{r-1} (m+1)^d, \quad\text{where}\quad
 m=\left\lceil\frac{e^{1/r}}{2} \sqrt{d} \varepsilon^{-1/r}\right\rceil.
 $$
 Lemmas~\ref{small derivatives} and \ref{even r lemma} yield
 $$
 e^{\rm app}\brackets{n,\C^r_d} 
 \leq E\brackets{Q_m^d,\C^r_d}
 \leq \frac{e d^{r/2}}{(2m)^r}
 \leq \varepsilon.
 $$
 Hence,
 $$
 n^{\rm app}\brackets{\varepsilon,\C^r_d} \leq n
 $$
 and this implies the result.
\end{proof}

To derive the upper bounds for odd $r$,
we use the following recursive formula.

\begin{lemma}
 \label{odd r lemma}
 Let $m\in \IN$, $d\in \IN$ and $r\in\IN$. Then
 \begin{equation*}
  E\brackets{Q_m^d,\C^r_d}
  \leq \frac{d}{2m}\,
  E\brackets{Q_m^d,\C^{r-1}_d}.
 \end{equation*}
\end{lemma}

\begin{proof}
 It suffices to show for any $\delta>0$ that
 \begin{equation*}
  E\brackets{Q_m^d,\C^r_d,\delta}
  \leq \delta^{2^{r-1}} + \frac{d}{2m}\,
  E\brackets{Q_m^d,\C^{r-1}_d,\delta} .
 \end{equation*}
 Letting $\delta$ tend to zero yields the statement.
 Let $f\in\C^r_d\brackets{Q_m^d,\delta}$ and
 let $z\in[0,1]^d$ such that $\abs{f(z)}=\norm{f}_\infty$.
 There is some $y\in Q_m^d$ such that $y$ and $z$
 are connected by an axis-parallel polygonal chain
 of length at most $d/(2m)$.
 For every $j\in\set{1,\dots,d}$, the partial derivative
 $\frac{\partial f}{\partial x_j}$ is 
 contained in $\C^{r-1}_d\brackets{Q_m^d,\delta}$.
 Integrating along the curve yields
 $$
 \abs{f(z)} \leq 
 \abs{f(y)} + \frac{d}{2m} \max\limits_{j=1...d} 
 \norm{\frac{\partial f}{\partial x_j}}_\infty
 \leq \delta^{2^{r-1}} + \frac{d}{2m} E\brackets{Q_m^d,\C^{r-1}_d,\delta}.
 $$
 This proves the lemma.
\end{proof}

Now, the upper bounds of Theorem~\ref{main theorem odd}
follow from the results for even $r$.
Note that the upper bound for $r=1$ is included.

\begin{proof}[Proof of Theorem~\ref{main theorem odd} (Upper Bound)]
 Let $d\in\IN$, $r\in\IN$ be odd and $\varepsilon>0$.
 For any $m\in\IN$,
 Lemma~\ref{even r lemma} and \ref{odd r lemma} yield
 $$
 E\brackets{Q_m^d,\C^r_d} \leq
  \frac{e d^{(r+1)/2}}{(2m)^r}.
 $$
 We set
 $$
 n=(d+1)^{r-1} (m+1)^d, \quad\text{where}\quad
 m=\left\lceil\frac{e^{1/r}}{2} d^{\frac{r+1}{2r}} \varepsilon^{-1/r}\right\rceil.
 $$
 We obtain
 $$
 e^{\rm app}\brackets{n,\C^r_d} 
 \leq E\brackets{Q_m^d,\C^r_d}
 \leq \varepsilon
 $$
 and hence
 $$
 n^{\rm app}\brackets{\varepsilon,\C^r_d} \leq n .
 $$
\end{proof}

We proceed similarly
to prove of the upper bound of Theorem~\ref{side theorem}.
For any $\delta>0$, any $r\in\IN_0$, $d\in\IN$ and $Q\subset [0,1]^d$,
we define the subclasses
\begin{equation*}
 \widetilde\C^r_d(Q,\delta) =
 \set{f\in\widetilde\C^r_d \mid 
 \abs{\partial_{\theta_1}\cdots\partial_{\theta_\ell} f(x)}
 \leq \delta^{2^{r-\ell-1}}
 \text{ for } x\in Q,  \ell < r, \theta_1\hdots\theta_\ell\in S_{d-1}}
\end{equation*}
and the auxiliary quantities
\begin{equation*}
 E\brackets{Q,\widetilde\C^r_d,\delta}=
 \sup\limits_{f\in \widetilde\C^r_d(Q,\delta)} \norm{f}_\infty
 \qquad\text{and}\qquad
 E\brackets{Q,\widetilde\C^r_d}= 
 \lim\limits_{\delta\downarrow 0} E\brackets{Q,\widetilde\C^r_d,\delta}
\end{equation*}
and obtain the following.

\begin{lemma}
\label{small derivatives 2}
Let $d,r\in \IN$ and $n\in\IN_0$.
If the cardinality of $Q\subset [0,1]^d$
is at most $n/(d+1)^{r-1}$, then
\begin{equation*}
 e^{\rm app}\brackets{n,\widetilde\C^r_d} 
 \leq E\brackets{Q,\widetilde\C^r_d}.
\end{equation*}
\end{lemma}

\begin{proof}
 Let $\delta \in (0,1)$.
 In the proof of Lemma~\ref{small derivatives}
 we constructed a point set $P$
 with cardinality at most $n$ such that
 any $f\in\C^r_d$ with $f\vert_P=0$ is contained
 in $\C^r_d(Q,\delta)$.
 In particular, any $f\in\widetilde\C^r_d$ with $f\vert_P=0$
 satisfies $\abs{D^\alpha f(x)}\leq \delta^{2^{r-\abs{\alpha}-1}}$
 for all $x\in Q$ and $\abs{\alpha} < r$.
 Taking into account that for $x\in [0,1]^d$, $\ell < r$ 
 and  $\theta_1\hdots\theta_\ell\in S_{d-1}$ we have
 $$
 \abs{\partial_{\theta_1}\cdots\partial_{\theta_\ell} f(x)}
 \leq d^{\ell/2} \max\limits_{\abs{\alpha}=\ell} \abs{D^\alpha f(x)} ,
 $$
 we obtain that $f\in \widetilde\C^r_d(Q,d^{\frac{r-1}{2}} \delta)$ and hence
 $$
 e^{\rm app}\brackets{n,\widetilde\C^r_d} \leq
  \sup\limits_{f\in \widetilde\C^r_d:\, f\vert_P=0} \norm{f}_\infty
  \leq \sup\limits_{f\in \widetilde\C^r_d(Q,d^{(r-1)/2} \delta)} \norm{f}_\infty
  = E\brackets{Q,\widetilde\C^r_d,d^{\frac{r-1}{2}}\delta}.
 $$
 Letting $\delta$ tend to zero yields the statement.
\end{proof}

For these classes, 
it is enough to consider the following
single-step recursion.

\begin{lemma}
 \label{recursion lemma 2}
 Let $m\in \IN$, $d\in\IN$ and $r\in\IN$. Then
 \begin{equation*}
  E\brackets{Q_m^d,\widetilde\C^r_d}
  \leq \frac{\sqrt{d}}{2m}\,
  E\brackets{Q_m^d,\widetilde\C^{r-1}_d}.
 \end{equation*}
\end{lemma}

\begin{proof}
 It suffices to show for any $\delta>0$ that
 \begin{equation*}
  E\brackets{Q_m^d,\widetilde\C^r_d,\delta}
  \leq \delta^{2^{r-1}} + \frac{\sqrt{d}}{2m}\,
  E\brackets{Q_m^d,\widetilde\C^{r-1}_d,\delta} .
 \end{equation*}
 To this end, let $f\in\widetilde\C^r_d\brackets{Q_m^d,\delta}$ and
 let $z\in[0,1]^d$ such that $\abs{f(z)}=\norm{f}_\infty$.
 There is some $y\in Q_m^d$ such that $y$ and $z$
 are connected by a line segment
 of length at most $\sqrt{d}/(2m)$.
 The directional derivative
 $\partial_\theta f$ with $\theta = \frac{z-y}{\norm{z-y}_2}$
 is contained in $\widetilde\C^{r-1}_d\brackets{Q_m^d,\delta}$.
 Integrating along the line yields
 $$
 \abs{f(z)} \leq 
 \abs{f(y)} + \frac{\sqrt{d}}{2m}
 \norm{\partial_\theta f}_\infty
 \leq \delta^{2^{r-1}} + \frac{\sqrt{d}}{2m} E\brackets{Q_m^d,\widetilde\C^{r-1}_d,\delta}.
 $$
 Letting $\delta$ tend to zero yields the statement.
\end{proof}

The upper bound of Theorem~\ref{side theorem}
can now be proven by induction on $r$.

\begin{proof}[Proof of Theorem~\ref{side theorem} (Upper Bound)]
Lemma~\ref{recursion lemma 2} and $E(Q_m^d,\widetilde\C^0_d)=1$ yield
$$
E\brackets{Q_m^d,\widetilde\C^r_d} \leq 
\brackets{\frac{\sqrt{d}}{2m}}^r
$$
for any $m\in\IN$, $d\in\IN$ and $r\in\IN_0$.
Let now $d\in\IN$, $r\in\IN$ and $\varepsilon>0$.
We set
 $$
 n=(d+1)^{r-1} (m+1)^d, \quad\text{where}\quad
 m=\left\lceil\frac{1}{2} \sqrt{d} \varepsilon^{-1/r}\right\rceil.
 $$
 Lemma~\ref{small derivatives 2} yields
 $$
 e^{\rm app}\brackets{n,\widetilde\C^r_d} 
 \leq E\brackets{Q_m^d,\widetilde\C^r_d}
 \leq \varepsilon
 $$
 and hence
 $$
 n^{\rm app}\brackets{\varepsilon,\widetilde\C^r_d} \leq n .
 $$
\end{proof}

\section{Lower bounds}

By Lemma~\ref{Bakhvalov}, we can estimate 
$e(n,\widetilde\C^r_d)$ from below as follows.
For any point set $P$ with cardinality at most $n$,
we construct a function $f\in \widetilde\C^r_d$ that vanishes on $P$
but has a large maximum in $[0,1]^d$,
a so called fooling function.
We will use the following lemma.
Note that
$$
 \norm{f}_{r,d} =
 \sup\limits_{\ell \leq r, \theta_i\in S_{d-1}}
 \Vert \partial_{\theta_1}\cdots\partial_{\theta_\ell} f \Vert_\infty
$$
defines a norm on the space of smooth functions $f:\IR^d\to \IR$ 
with compact support.

\begin{lemma}
\label{radial function}
 There exists a sequence $\brackets{g_d}_{d\in\IN}$ of smooth functions 
 $g_d:\IR^d\to \IR$
 with compact support in the unit ball that satisfy $g_d(0) =1$ and
 \begin{equation*}
  \sup\limits_{d\in\IN}\, \norm{g_d}_{r,d}\, < \infty
  \qquad \text{for all } \quad r\in\IN_0.
 \end{equation*}
\end{lemma}

\begin{proof}
 Take any function $h\in\C^\infty(\IR)$ which equals one on $(-\infty,0]$ and zero on $[1,\infty)$.
 Then the radial functions
 \begin{equation*}
  g_d: \IR^d \to \IR, \quad g_d(x)=h\brackets{\norm{x}_2^2}
 \end{equation*}
 for $d\in\IN$ have the desired properties.
 This follows from the fact that the directional derivative
 $\partial_{\theta_1}\cdots\partial_{\theta_r} g_d(x)$
 only depends on the length of $x$ and the angles
 between each pair of vectors $\theta_1,\hdots,\theta_r\in S_{d-1}$ and $x\in\IR^d$. 
 As soon as $d$ is large enough 
 to enable all constellations of lengths and angles,
 the norm $\norm{g_d}_{r,d}$ is independent of $d$.
\end{proof}

To obtain a suitable fooling function for a given point set $P$,
it is enough to shrink and shift the support of $g_d$
to the largest euclidean ball that does not intersect with $P$.
The radius of this ball can be estimated by a simple volume argument.

\begin{lemma}
\label{radius lower bound}
 Let $d\in\IN$ and $P\subset[0,1]^d$ with cardinality $n\in\IN$.
 Then there exists a point $z\in[0,1]^d$ with
 \begin{equation*}
  \dist_2\brackets{z,P} \geq \frac{\sqrt{d}}{5 n^{1/d}}.
 \end{equation*}
\end{lemma}

\begin{proof}
 The set
 \begin{equation*}
  B_R(P) = \bigcup_{p \in P} B_R(p)
 \end{equation*}
 of points within a distance $R>0$ of $P$ has the volume
 \begin{equation*}
  \vol_d\brackets{B_R(P)}
  \leq n R^d\, \vol_d\brackets{B_1(0)}
  = \frac{n R^d\, \pi^{d/2}}{\Gamma\brackets{\frac{d}{2}+1}}.
 \end{equation*}
 By Stirling's Formula, this can be estimated from above by
  \begin{equation*}
  \vol_d\brackets{B_R(P)}
  \leq \frac{n R^d\, \pi^{d/2}}{\frac{\sqrt{2\pi}}{e} \brackets{\frac{d}{2e}}^{d/2}}
  \leq \brackets{ \frac{n^{1/d} e^{3/2}}{\sqrt{d}}\, R }^d
 .\end{equation*}
 If $R=\frac{\sqrt{d}}{5 n^{1/d}}$, the volume is less than one
 and $[0,1]^d\setminus B_R(P)$ must be nonempty.
\end{proof}

We are ready to prove the lower bound of Theorem~\ref{side theorem}.

\begin{proof}[Proof of Theorem~\ref{side theorem} (Lower Bound)]
Let $r\in\IN$, $d\in\IN$ and $n\in\IN$.
Let $P$ be any subset of $[0,1]^d$ with cardinality at most $n$.
Let $g_d$ be like in Lemma~\ref{radial function} and set
 \begin{equation*}
  K_r=\sup\limits_{d\in\IN}\, \norm{g_d}_{r,d}
  \qquad\text{and}\qquad
  R=\min\set{1,\frac{\sqrt{d}}{5 n^{1/d}}}.
 \end{equation*}
By Lemma~\ref{radius lower bound} there is a point $z\in[0,1]^d$
such that $B_R(z)$ does not contain any element of $P$.
Hence, the function
 \begin{equation*}
  f_*: [0,1]^d \to \IR, \quad f_*(x) = \frac{R^r}{K_r}\, g_d\brackets{\frac{x-z}{R}}
 \end{equation*}
is an element of $\widetilde\C^r_d$ and vanishes on $P$. 
We obtain
\begin{equation*}
 \sup\limits_{f\in \widetilde\C^r_d:\, f\vert_P=0} \norm{f}_\infty
 \geq \norm{f_*}_\infty
 \geq f_*(z)
 = \frac{R^r}{K_r}
 = \min\set{\frac{1}{K_r},\frac{d^{r/2}}{5^r K_r n^{r/d}}}
.\end{equation*}
Since this is true for any such $P$, Lemma~\ref{Bakhvalov} yields 
\begin{equation}
\label{error numbers lower bound}
 e^{\rm app}\brackets{n,\widetilde\C^r_d} \geq
 \min\set{\frac{1}{K_r},\frac{d^{r/2}}{5^r K_r n^{r/d}}}.
\end{equation}
We set $\varepsilon_r=1/K_r$.
Given $\varepsilon\in(0,\varepsilon_r)$, 
the right hand side in \eqref{error numbers lower bound}
is larger than $\varepsilon$
for any $n$ smaller than $d^{d/2}/(5^d K_r^{d/r}\varepsilon^{d/r})$.
This yields
\begin{equation*}
 n^{\rm app}\brackets{\varepsilon,\widetilde\C^r_d}
 \geq \brackets{(5^rK_r)^{-1/r} \sqrt{d}\, \varepsilon^{-1/r}}^d
\end{equation*}
as it was to be proven.
\end{proof}

In the same way, we obtain lower bounds
for the case that the domains $[0,1]^d$ are replaced
by other domains $D_d\subset\IR^d$ that satisfy
$\vol_d(D_d)\geq a^d$ for some $a>0$ and all $d\in\IN$.
We simply have to multiply the radii in the previous proofs by $a$.

We now turn to the lower bounds of Theorem~\ref{main theorem}
and \ref{main theorem odd}.

\begin{proof}[Proof of Theorem~\ref{main theorem} and \ref{main theorem odd} (Lower Bounds)]
Note that $\C^r_d$ contains $\widetilde\C^r_d$ and hence
\begin{equation*}
 n^{\rm app}\brackets{\varepsilon,\C^r_d}
 \geq
 n^{\rm app}\brackets{\varepsilon,\widetilde\C^r_d}.
\end{equation*}
Furthermore, any $\varepsilon$-approximation
of a function on $[0,1]^d$ immediately yields an
$\varepsilon$-approximation of its integral and hence
\begin{equation*}
 n^{\rm app}\brackets{\varepsilon,\C^r_d}
 \geq
 n^{\rm int}\brackets{\varepsilon,\C^r_d}.
\end{equation*}
With these relations at hand,
the desired lower bounds for $r\geq 2$ immediately
follow from Theorem~\ref{side theorem}.
The lower bound for $r=1$ follows from
the complexity of numerical integration as
studied by Hinrichs, Novak, Ullrich and Woźniakowski~\cite{hnuw}.
\end{proof}

\begin{ack}
 I wish to thank Erich Novak for many fruitful discussions
 in the context of this paper.
\end{ack}

\raggedright{

}

\end{document}